\documentclass[12pt]{amsart}
\usepackage{xcolor}
\pagecolor{black}
\color{white}
\usepackage[T1]{fontenc}
\usepackage[utf8]{inputenc}
\usepackage{color,amssymb,amsmath,mathrsfs, enumerate,esint}
\usepackage{a4wide}
\usepackage{hyperref}
\usepackage{pgf,tikz,pgfplots}
\pgfplotsset{compat=1.15}
\usepackage{mathrsfs}
\usepackage{cite}
\usetikzlibrary{arrows}
\theoremstyle{plain}
\newtheorem{Theorem}{Theorem}[section]
\newtheorem{Lemma}[Theorem]{Lemma}

\newtheorem{Example}[Theorem]{Example}
\theoremstyle{definition}

\newtheorem{Definition}{Definition}[section]
\newtheorem{Pro}{Problem}

\DeclareMathOperator{\er}{\mathbb{R}}

\DeclareMathOperator{\supp}{supp}
\title{Higher order pointwise estimates for intermediate derivative}

\author{Miroslav Hor\'ak}

\author{Filip Soudsk\'y}

\address{F.~Soudsk\'y: University of Chemistry and Technology, Prague, Technick\'a 5, 166 28 Praha 6}
\email{soudskyf@vscht.cz}

\author{Martina \v Sim\r unkov\'a}
\address{M.~\v Sim\r unkov\'a: Technical University of Liberec, Liberec, Studentsk\'a 1402/2, 461 17 Liberec, Czech Rebublic}
\email{martina.simunkova@tul.cz}

\address{M. Horák: Technical University of Liberec, Liberec, Studentská 1402/2, 461 17 Liberec, Czech Rebublic}
\email{miroslav.horak1@tul.cz}

\begin{document}
\maketitle
\begin{abstract}
    The estimates of intermediate derivative play important role in the theory of partial differential equations. The modern approach to this problem is via pointwise estimates, since this allows to obtain results for various function norms. The paper aims to extend the pointwise estimate for intermediate derivative in terms of the averaging operator for higher orders derivatives.
\end{abstract}
\section{Introduction}
The interpolation inequalities have been used as a key instrument in partial differential equations and mathematical modelling (see for instance, \cite{Lady} or \cite{N}). One of the classical examples of this type of inequality is the famous Gagliardo-Nirenberg inequality. The first special case of this inequality can be traced back to 1913 (see \cite{Landau1913}). Later, the multidimensional version was first introduced by Kolmogorov in 1937. The special cases were also studied by John Nash \cite{N} and Olga Ladyzenskaya\cite{Lady} both of the works are using this inequality as an instrument for studiing existence of solution for certain PDEs. The general result for Lebesgue spaces was finally established independently in \cite{Ni} and \cite{Gag}. Since more general type of spaces such as Orlicz spaces, Lorentz spaces etc. appeared, the more general form of the inequality was needed. Since the classical methods used by Gagliardo and Nirenberg fail to achieve estimates for more general norms a new tool was needed. Pointwise inequalities proved to be the most fruitful approach to the problem. The initial result in this field was achieved by Agnieszka Kalamajska in 1993 (see \cite{Kal}). However, the estimate was achieved only in term of the full gradient inequality. The results for pure partial derivatives cannot be reconstructed since the approach requires potential estimates (this is a principle problem as the counter-example in our paper shows). Since the maximal operators are involved in the result, we need to have its boundedness on the function spaces we are working on. The problem was fixed in \cite{LRS2}. The authors replaced the maximal operator by an averaging operator over a family of set with bounded overlapping (designed for each function individually but with an uniform constant of overlaps). A similar approach is known from several results in harmonic analysis as a sparse domination approach. Such an operator has the key advantage, which is the boundedness on arbitrary r.i. Banach function space. The estimate also works for pure partial derivative. However, the pointwise result was given only for the combination of the first and second order derivative (the norm results of higher order were achieved using double induction). 

In our paper we are going to present the pointwise inequality covering  some of the higher order derivatives. We also show why the estimate of the pure derivative by the maximal operator is impossible to achieve. We also formulate some open problems in the end of the paper.

\section{Main results}
Our main result is the following extension of the result obtained in \cite{LRS2} and reads as follows.
\begin{Theorem}\label{MR}
Let $0<j$ be a natural number. There exists a constant $C=C(j)$, such that for all $u\in \mathcal{C}^{j+1}_c(\er^n)$ there exists a family of sets $\mathcal{P}=\mathcal{P}(u)$ such that the following inequalities hold
\begin{equation}
    \left|\frac{\partial^j u}{\partial x_i^j}\right|\leq C \left(A_{\mathcal{P}}\left(\frac{\partial^{j+1} u}{\partial x_i^{j+1}}\right)\right)^{\frac{j}{j+1}}(A_{\mathcal{P}}u)^{\frac{1}{j+1}}
\end{equation}
and
$$
\sum_{A\in\mathcal{P}}\chi_A\leq 5.
$$
\end{Theorem}

We also show that the result obtained by A. Kalamajska concerning the bounds for intermediate gradient cannot be obtained for pure partial derivatives. Moreover, instead of using the Hardy–Littlewood maximal operator, we will use the Sparse averaging operator (see \ref{Sparse} for precise definition)

\section{Preliminary results}

\begin{Definition}\label{Sparse}
    Let \(\mathcal{S}\) be a system of measurable sets of finite measure and let \(f\) be a locally integrable function on \(\er^n\). Define the mean value operator over \(\mathcal{S}\) by
\[
A_\mathcal{S}f(x):=\sum_{A\in\mathcal{S}}\chi_A(x)\fint_A |f(s)|\,ds = \sum_{A\in\mathcal{S}}\frac{1}{|A|}\chi_A(x)\int_A |f(s)|\,ds.
\]

\end{Definition}

A classical theorem from \cite{BS} states that for a disjoint system of measurable sets $\mathcal{S}$ an arbitrary Banach function space one has
$$
\|A_{\mathcal{S}}\|_{X\to X}=1
$$

This result may be easily extended for sets with controlled overlaps. Let us state this in the more precise form. In the following text we shall use notation
\begin{equation}\label{Levelset}
\{f>\alpha\}:=\{x\in\er^n: f(x)>\alpha\}   
\end{equation}
for level sets of a functions, with obvious modifications for $\{f\geq\alpha\}, \{f\leq \alpha\}, \{f<\alpha\}$.
\begin{Theorem}
Let $X$ be an arbitrary r.i. Banach function space. And let $\mathcal{S}$ be a system of intervals for which
\begin{equation}\label{OC}
    \sum_{A\in\mathcal{S}}\chi_A(x)\leq C
\end{equation}
Then we have
\begin{equation}
\|A_{\mathcal{S}}u\|_X\leq C\|u\|_X    
\end{equation}

\end{Theorem}

The proof of this theorem follows easily from the fact that operator is bounded both on $L^1$ and $L^\infty$. For details see \cite[Lemma 3.2]{LRS2}.

\begin{Lemma}[An estimate by a lower derivative]\label{ODL}
    Let $k\in\mathbb{N}$ and $I=[a,b]$ be a bounded interval. Then there exists a constant $C=C(k)$ such that
    for every $f\in \mathcal{C}^k(a,b)\cap\mathcal{C}([a,b])$ it holds 
    \begin{equation}\label{DO}
        (b-a)^{k+1}\inf_{\tau\in (a,b)}|f^{(k)}(\tau)|\leq C\int_a^b |f(s)|ds.
    \end{equation}
\end{Lemma}
\begin{proof}
    We will prove this by induction on $k$. First, for $k=0$, the statement is obvious.
    Let us suppose that the statement holds for $k\in\mathbb{N}$, with constant $C(k)$.
    We note that 
    $$
    \max\left\{\inf_{x\in(a,\frac{3a+b}{4})}|f^{(k)}(x)|,\inf_{x\in(\frac{a+3b}{4},b)}|f^{(k)}(x)|\right\}\geq \inf_{x\in(a,b)}|f^{(k+1)}(x)|\frac{b-a}{4}.
    $$
    Consequently, we obtain
    $$
    \begin{aligned}
    (b-a)^{k+2}\inf_{\tau\in (a,b)}|f^{(k+1)}(\tau)| &
    \leq4(b-a)^{k+1}
    \max\left\{\inf_{x\in(a,\frac{3a+b}{4})}|f^{(k)}(x)|,\inf_{x\in(\frac{a+3b}{4},b)}|f^{(k)}(x)|\right\}
    \\&
    = 4^{k+2}
    \left(\frac{b-a}{4}\right)^{k+1}
    \max\left\{\inf_{x\in(a,\frac{3a+b}{4})}|f^{(k)}(x)|,\inf_{x\in(\frac{a+3b}{4},b)}|f^{(k)}(x)|\right\}
    \\&
    \leq 4^{k+2}C(k)\int_a^b |f(s)|ds.
    \end{aligned}
    $$
    Hence, if we define $C(k+1):=4^{k+2}C(k)$ the inequality holds for $k+1$ with this constant.
\end{proof}
First, we will look at the analogous case of the theorem \ref{MR} in one dimension.

\begin{Lemma}[One-dimensional estimate]\label{ODE}
Let $j$ be a natural number. There exists a positive constant $C(j)$, such that for every \(u \in \mathcal{C}^{j+1}\cap (L^1+L^\infty)(\mathbb{R})\) there exists a countable family \(\mathcal{P}=\mathcal{P}(u)\) of bounded open intervals satisfying
\begin{equation}\label{IE}
\left|u^{(j)}\right|^{j+1}\leq C(j)\sum_{P\in\mathcal{P}}\left(\fint_{P}|u^{(j+1)}(s)|\,ds\right)^{j}\left(\fint_{P}|u(s)|\,ds\right)\chi_{P}
\end{equation}
and
\begin{equation}\label{CH}
\sum_{P \in \mathcal{P}} \chi_P \leq 3.
\end{equation} 
\end{Lemma}

\begin{proof}
Define the disjoint covering of \(\{u^{(j)}\neq 0\}\) by
\[
E_l:=\{ 2^{l-1} \leq |u^{(j)}| < 2^l\}.
\]
Observe 
\[
\bigcup_{l\in \mathbb{Z}} E_l = \{u^{(j)}\neq 0\},
\]
and
\[E_l\cap E_k=\emptyset,\]
for $k\neq l$.

Next, for \(x\in E_l\) define the points \(y(x)\) and \(z(x)\) by
\[
y(x) = \sup\{\tau \in \mathbb{R} : (x,\tau) \subset E_{l-1} \cup E_l \cup E_{l+1}\}
\]
and
\[
z(x) = \inf\{\tau \in \mathbb{R} : (\tau,x) \subset E_{l-1} \cup E_l \cup E_{l+1}\}.
\]
Now denote
\[
\mathcal{P} = \{ (z(x), y(x))\} \quad x \in \{u^{(j)}\neq0 \}.
\]
Note that $y(x),z(x)$ are finite for every $x\in E_l$ (for arbitrary $l\in\mathbb{Z}$). Indeed, by integral form of the remainder
$$
u(s)=\sum_{i=0}^{j-1}u^{(i)}(x)\frac{(s-x)^i}{i!}+\int_x^s u^{(j)}(t)\frac{(s-t)^{j-1}}{(j-1)!}dt=T^{j-1}_{u,x}(s)+\int_x^s u^{(j)}(t)\frac{(s-t)^{j-1}}{(j-1)!}dt,
$$
where $T^{j-1}_{u,x}$ stands for the Taylor polynomial. Without loss of generality, we may assume $u^{(j)}(x)>0$. For $s\in(x,y(x))$ estimate
$$
\begin{aligned}
|u(s)|&\geq \left|\int_x^s u^{(j)}(t)\frac{(s-t)^{j-1}}{(j-1)!}dt\right|-|T^{j-1}_{u,x}(s)|\\
&\geq2^{l-2}\left|\int_x^s\frac{(s-t)^{j-1}}{(j-1)!}dt\right|-|T^{j-1}_{u,x}(s)|\\
&=2^{l-2}\frac{(s-x)^j}{j!}-|T^{j-1}_{u,x}(s)|
\end{aligned}
$$
Now, if we compute a limit, we get the following
\[\begin{aligned}
  \lim_{s \rightarrow \infty} |u(s)| &\geq\lim_{s \rightarrow \infty} \left(\frac{2^{l-2}}{j!}(s-x)^j\right)-|T^{j-1}_{u,x}(s)|\\
  &=\infty
\end{aligned}
\]
So the limit equals infinity however, that would mean that \(u \notin L^1+L^{\infty}\). This contradicts our assumption. So it is true that \(y(x),z(x) \in \mathbb{R}\).
 
Now we prove the second part of the lemma
\[\sum_{P \in \mathcal{P}} \chi_P \leq 3.\]
It is clear that by defining \(\mathcal{P}\) in this way, any point can belong to at most three intervals of the family \(\mathcal{P}\). This is since if we choose two arbitrary points \(s,t \in E_l\), then one of two cases occurs:
\[
(z(s), y(s)) = (z(t), y(t)),
\]
or
\[
(z(s), y(s)) \cap (z(t), y(t)) = \emptyset.
\]
Moreover, if \(t\in E_l\), then it must be that \(t\in (y(s), z(s))\) only under the condition that \(s\in E_{l-1}\cup E_l \cup E_{l+1}\), which implies that \(\eqref{CH}\) holds.
It remains to prove \(\eqref{IE}\) on the set \(\{u^{(j)}\neq0\}\).
For \(x \in E_l\) and for every \(t \in (y(x), z(x))\) we have
\[
2^{l-1} \leq |u^{(j)}(x)| < 2^l \quad \text{and} \quad 2^{l-2} \leq |u^{(j)}(t)| < 2^{l+1}.
\]
By using Lemma \ref{ODL}, we obtain the following estimate:
\begin{equation}\label{PČ}
|u^{(j)}(x)| \leq 4\,\inf_{t \in (z,y)}|u^{(j)}(t)| \leq \frac{C}{(y-z)^{j+1}}\int_{z}^{y}|u(s)|\,ds.
\end{equation}
Now, we need to estimate the value of the derivative in terms of the higher-order derivative. For this purpose, we use the following inequality (note that $|u^{j}(z)|\in\{2^{l-2},2^{l+1}\}$):
\begin{equation}\label{OHD}
|u^{(j)}(x)| \leq 2|u^{(j)}(x)-u^{(j)}(z)| \leq 2\int_z^{x}|u^{(j+1)}(s)|\,ds \leq 2\int_{z}^{y}|u^{(j+1)}(s)|\,ds.
\end{equation}
Raising the previous inequality to the power \(j\) and then multiplying it with inequality \eqref{PČ}, we obtain

\begin{equation}\label{1}
\left|u^{(j)}(x)\right|^{j+1}\leq C(j)\left(\fint_{z}^y|u^{(j+1)}(s)|\,ds\right)^{j}\left(\fint_{z}^y|u(s)|\,ds\right),
\end{equation}
with $C(j)=2^jC$. From here the desired inequality follows.
\end{proof}

\subsection{Construction of the family $\mathcal{S}$}

Now, we need to construct the system of sets in higher dimensions, with the desired properties.

 Let $k\in \mathbb{Z}$ we define $\eta_k>0$ as the number satisfying with $C(j)$ from (\ref{IE})
 \begin{multline}\label{Modulus}
   \forall x,\hat x\in\left[0,\left\|\frac{\partial^{j+1}u}{\partial x_1^{j+1}}\right\|_\infty\right], \forall y,\hat y\in [0,\|u\|_{\infty}]\quad 
   \\
   \max\{|x-\hat x|,|y-\hat y|\}\le \eta_k\implies \left|x^{\frac{j}{j+1}}y^{\frac{1}{j+1}}-\hat x^{\frac{j}{j+1}}\hat y^{\frac{1}{j+1}}\right|\leq \frac{2^{(k-3)(j+1)}}{C(j)}  
 \end{multline}
 
Note that the definition is correct since the function $f(x,y)=x^{\frac{j}{j+1}}y^{\frac{1}{j+1}}$ is uniformly continuous on any compact subset of $\er^2$.

For arbitrary \(k\in\mathbb{Z}\), by the uniform continuity of \(u\) and its derivatives up to order \(j+1\), there exists \(\delta_k>0\) such that \(\lvert x-y\rvert\le \delta_k\) implies

\begin{equation}\label{Delta}
\max\Bigl\{
    \lvert u(x)-u(y)\rvert,\;
    \Bigl\lvert\tfrac{\partial^j u(x)}{\partial x_1^j}-\tfrac{\partial^j u(y)}{\partial x_1^j}\Bigr\rvert,\;
    \Bigl\lvert\tfrac{\partial^{j+1} u(x)}{\partial x_1^{j+1}}-\tfrac{\partial^{j+1} u(y)}{\partial x_1^{j+1}}\Bigr\rvert
\Bigr\}
\le
\min\left\{2^{k-3},\eta_k\right\},
\end{equation}

Let \(u\in \mathcal{C}^{j+1}_c(\mathbb{R}^n)\). For \(x \in \mathbb{R}^n\) choose \(\bar{x}\in \mathbb{R}^{n-1}\), where \(x=(x_i,\bar{x})\). To simplify the notation we will only consider the case \(i=1\), since other cases can be done analogously. Without loss of generality, we find covering only for set \(\left\{\Bigl|\frac{\partial^j u}{\partial x_1^j}\Bigr|>0\right\}\). For each \(k\in\mathbb{Z}\) define
\[
E_k
=\Bigl\{\ 2^{k-1}\le \Bigl|\tfrac{\partial^j u}{\partial x_1^j}\Bigr|<2^k\Bigr\}.
\]
Observe, that
\[\bigcup_kE_k=\left\{\left|\frac{\partial^j u}{\partial x_1^j}\right|>0\right\}\]
and moreover, the sets \(E_k\) are pairwise disjoint. As in the one-dimensional case, we define for $x=(x_1,\bar{x})\in E_k$

\begin{align}
\label{yx}
 y(x)
&= \sup \left\{t\in\mathbb{R}:(x_1,t)\subset \left\{\,2^{k-2}\le \left|\tfrac{\partial^j u(t,\bar x)}{\partial x_1^j}\right|<2^{k+1}\right\}\right\},
\\
\label{zx}
 z(x)
&= \inf \left\{t\in\mathbb{R}:(t,x_1)\subset \left\{\,2^{k-2}\le \left| \tfrac{\partial^j u(t,\bar x)}{\partial x_1^j}\right|<2^{k+1}\right\}\right\},
\end{align}

For $k\in\mathbb Z$, and $\bar{x}\in\mathbb R^{n-1}$, $t\in\er$ such that $(t,\bar x)\in E_k$ we define a cylinder 
\begin{equation}\label{Cyl}
\begin{aligned}
R(t,\bar x)&:=(z(t,\bar x),y(t,\bar x))\times B(\bar x ,\delta_{k})
\end{aligned}
\end{equation}

For this cylinders let us prove the following lemmata.

\medskip

\begin{Lemma}\label{CoV}
    Let $x=(t,\bar x)\in E_k$ then 
    $$
    R(t,\bar x)\subset E_{k-2}\cup E_{k-1}\cup E_k\cup E_{k+1}\cup E_{k+2}
    $$
\end{Lemma}

\noindent
\begin{proof}
Let \(y=(y_1,\bar{y})\in R_k(t,\bar{x})\). Set \(x:=(y_1,\bar{x})\). By the definition of the interval 
\((z(t,\bar{x}),y(t,\bar{x}))\), we have
\[
2^{k-2}\le \left|\tfrac{\partial^j u(x)}{\partial x_1^j}\right|<2^{k+1}.
\]
Moreover, since \(|\bar{x}-\bar{y}|<\delta_k\), we have \(|x-y|<\delta_k\). Hence, by the choice of \(\delta_k\),
\[
\Bigl|\tfrac{\partial^j u(x)}{\partial x_1^j}-\tfrac{\partial^j u(y)}{\partial x_1^j}\Bigr|\le 2^{k-3}.
\]
Therefore,
\[
\left|\tfrac{\partial^j u(y)}{\partial x_1^j}\right|
\ge
\left|\tfrac{\partial^j u(x)}{\partial x_1^j}\right|-2^{k-3}
\ge 2^{k-2}-2^{k-3}
=2^{k-3},
\]
and similarly
\[
\left|\tfrac{\partial^j u(y)}{\partial x_1^j}\right|
\le
\left|\tfrac{\partial^j u(x)}{\partial x_1^j}\right|+2^{k-3}
<2^{k+1}+2^{k-3}
<2^{k+2}.
\]
Hence \(y\in E_{k-2}\cup\cdots\cup E_{k+2}\)
\end{proof}

\medskip

\begin{Lemma}
Let $(t,\bar x)\in E_k$. Then there exists a constant $C_{1D}$ independent of $k$, such that for \(\bar v\in B(\bar x,\delta_k)\) the following inequality holds.
\begin{equation}
\label{onedimodhadinz}
2^{k}\le
C_{1D}
\left(\fint_{z(t,\bar x)}^{y(t,\bar x)} \left|\frac{\partial^{j+1} u}{\partial x_1^{j+1}}(t,\bar v) \right|\, dt\right)^{\frac{j}{j+1}}
\left(\fint_{z(t,\bar x)}^{y(t,\bar x)} \left| u(t,\bar v) \right| \, dt\right)^{\frac{1}{j+1}}
\end{equation}
\end{Lemma}
\begin{proof}
%    Let us choose arbitrary $(t,\bar y)\in R_k=(z(\bar x),y(\bar x))\times B(\bar x, \delta_k)$. 
Using the properties of $\delta_k$ (see\eqref{Modulus} and \eqref{Delta}), we obtain the following 
\[
\begin{aligned}
&2^{k-2}\leq\left|\frac{\partial^j u}{\partial x^j_1}(t,\bar x)\right|
\leq \left(C(j)\left(\fint_{z(t,\bar x)}^{y(t,\bar x)} \left|\frac{\partial^{j+1} u}{\partial x_1^{j+1}}(t,\bar x) \right|\, dt\right)^j
\fint_{z(t,\bar x)}^{y(t,\bar x)} \left| u(t,\bar x) \right| \, dt\right)^{\frac{1}{j+1}}
\\
&\leq \left(C(j)\left(\left(\fint_{z(t,\bar x)}^{y(t,\bar x)} \left|\frac{\partial^{j+1} u}{\partial x_1^{j+1}}(t,\bar v) \right|\, dt\right)^j
\fint_{z(t,\bar x)}^{y(t,\bar x)} \left| u(t,\bar v) \right| \, dt+\frac{2^{(k-3)(j+1)}}{C(j)}\right)\right)^{\frac{1}{j+1}}
\\
&= \left(C(j)\left(\fint_{z(t,\bar x)}^{y(t,\bar x)} \left|\frac{\partial^{j+1} u}{\partial x_1^{j+1}}(t,\bar v) \right|\, dt\right)^j
\fint_{z(t,\bar x)}^{y(t,\bar x)} \left| u(t,\bar v) \right| \, dt + 2^{(k-3)(j+1)}\right)^{\frac{1}{j+1}}
\\
&\leq \left(C(j)\left(\fint_{z(t,\bar x)}^{y(t,\bar x)} \left|\frac{\partial^{j+1} u}{\partial x_1^{j+1}}(t,\bar v) \right|\, dt\right)^j
\fint_{z(t,\bar x)}^{y(t,\bar x)} \left| u(t,\bar v) \right| \, dt\right)^{\frac{1}{j+1}} + 2^{k-3}
\end{aligned}
\]
And therefore the estimate \eqref{onedimodhadinz} holds with $C_{1D}=8C(j))^{1/(j+1)}$.
\end{proof}

\noindent
Now let us define
\begin{equation}
    R_k:=\bigcup_{(t,\bar x)\in E_k}R(t,\bar x)
\end{equation}

\noindent
Now by the lemma \ref{CoV} we see that 
\begin{equation}\label{OLC}
\sum_{k}\chi_{R_k}\leq 5
\end{equation}
holds. Let us also set the following notation.
$$
\bar R_k:=\{\bar x\in\mathbb{R}^{n-1}:\exists t\in\er: (t,\bar x)\in R_k\}
$$
and 
$$
R_k(\bar x):=\{t:(t,\bar x)\in R_k\}
$$
Note that since $R_k$ is open, for every $\bar x$ $R_k(\bar x)$ is open as well and hence it may be written as a union of disjoint open intervals, more precisely
$$
R_k(\bar x)=\bigcup_{l}I_l^k(\bar x)
$$

\section{Proof of the main theorem}

\begin{proof}[Proof of Theorem \ref{MR}]
Now, using the H\" older inequalities we can estimate
\begin{equation}\label{ODHAD}
\begin{aligned}
&\left(\fint_{R_k}\left|\frac{\partial^{j+1} u}{\partial x_1^{j+1}}(x)\right|dx\right)^{\frac{j}{j+1}}\left(\fint_{R_k}|u(x)|dx
\right)^{\frac{1}{j+1}}
\\
&=\frac{1}{|R_k|}  \left(\int_{\bar R_k}\int_{R_k(\bar x)}\left|\frac{\partial^{j+1}u}{\partial x_1^{j+1}}(t,\bar x)\right|dt d\bar x\right)^{\frac{j}{j+1}}\left(\int_{\bar R_k}\int_{R_k(\bar x)}|u(t,\bar x)|dtd\bar x\right)^{\frac{1}{j+1}}\\
&\geq \frac{1}{|R_k|}\int_{\bar R_k}\left(\int_{R_k(\bar x)}\left|\frac{\partial^{j+1}u}{\partial x_1^{j+1} 
}(t,\bar x)\right|dt\right)^{\frac{j}{j+1}}\left(\int_{R_k(\bar x)}|u(t,\bar x)|dt\right)^{\frac{1}{j+1}}d\bar x\\
&= \frac{1}{|R_k|}\int_{\bar R_k}\left(\sum_l\int_{I_l^k(\bar x)}\left|\frac{\partial^{j+1} u}{\partial x_1^{j+1} }(t, \bar x)\right|dt\right)^{\frac{j}{j+1}}\left(\sum_l\int_{I_l^k(\bar x)}|u(t, \bar x)|dt\right)^{\frac{1}{j+1}}d\bar x\\
&\geq \frac{1}{|R_k|}\int_{\bar R_k}\sum_l\left(\int_{I_l^k(\bar x)}\left|\frac{\partial^{j+1}u}{\partial x_1^{j+1} }(t,\bar x)\right|dt\right)^{\frac{j}{j+1}}\left(\int_{I_l^k(\bar x)}|u(t, \bar x)|dt\right)^{\frac{1}{j+1}}d\bar x
%\\
%&=\frac{1}{|R_k|}\int_{\bar R_k}\sum_l |I_l^k(\bar x)|\left(\fint_{I_l^k(\bar x)}\left|\frac{\partial^{j+1}u}{\partial x_1^{j+1} }\right|dt\right)^{\frac{j}{j+1}}\left(\fint_{I_l^k(\bar x)}|u(x)|dt\right)^{\frac{1}{j+1}}d\bar x\\
\end{aligned}
\end{equation}
Let us fix $k\in\mathbb Z, l\in \mathbb{N}$. And denote
$$
(a,b):=I_l^k(\bar x)
$$
Now, the interval
\begin{equation}\label{INT}
\left[\frac{3a+b}{4},\frac{a+3b}{4}\right]
\end{equation}
is compact and the family of open intervals
\[%\begin{equation}\label{SJ}
I_l^k(\bar x)=\bigcup_{\bar u\in B(\bar x,\delta_k)\wedge (t, \bar u)\in E_k\wedge t\in I_l^k(\bar x)}(y(t,\bar u), z(t,\bar u))
\] 
covers inteval \eqref{INT}, therefore we may pick its finite subcovering 
$$
\bigcup_{i=1}^n(z(u_i),y(u_i)) \supset [(3a+b)/4,(a+3b)/4]
$$
where $u_i\in E_k$ and the numbers $z(u_i)$ and $y(u_i)$ are constructed using the lemma \ref{ODE} and precise definition is given in \eqref{yx} and \eqref{zx}. % more
%$$
Moreover we may assume that the families 
$$
J_{2i}=(z(u_{2i}),y(u_{2i}))\textup{ and } J_{2i+1}=(z(u_{2i+1}), y(u_{2i+1}))
$$
are pairwise disjoint. Note that for one of these families (which we denote by $\mathcal{J}$) the following measure estimate holds
$$
\left|\bigcup \mathcal{J}\right|\geq \frac{b-a}{4}
$$
Let us denote intervals in the family $\mathcal{J}$ by
$(z(v_{i}),y(v_i))$. Now, let us continue the estimate \eqref{ODHAD}.
$$
\begin{aligned}
&\frac{1}{|R_k|}\int_{\bar R_k}\sum_l \left(\int_{I_l^k(\bar x)}\left|\frac{\partial^{j+1}u}{\partial x_1^{j+1}}(t, \bar x)\right|dt\right)^{\frac{j}{j+1}}\left(\int_{I_l^k(\bar x)}|u(t, \bar x)|dt\right)^{\frac{1}{j+1}}d\bar x
\\
&\geq
\frac{1}{|R_k|}\int_{\bar R_k}\sum_l \left(\sum_{i}\int_{z(v_i)}^{y(v_i)}\left|\frac{\partial^{j+1}u}{\partial x_1^{j+1}}(t,\bar x)\right|dt\right)^{\frac{j}{j+1}}\left(\sum_{i}\int_{z(v_i)}^{y(v_i)}|u(t,\bar x)|dt\right)^{\frac{1}{j+1}}d\bar x
\\
&\geq
\frac{1}{|R_k|}\int_{\bar R_k}\sum_l\sum_{i} \left(\int_{z(v_i)}^{y(v_i)}\left|\frac{\partial^{j+1}u}{\partial x_1^{j+1}}(t,\bar x)\right|dt\right)^{\frac{j}{j+1}}\left(\int_{z(v_i)}^{y(v_i)}|u(t,\bar x)|dt\right)^{\frac{1}{j+1}}d\bar x
\\
&=
\frac{1}{|R_k|}\int_{\bar R_k}\sum_l\sum_{i} (y(v_i)-z(v_i))\left(\fint_{z(v_i)}^{y(v_i)}\left|\frac{\partial^{j+1}u}{\partial x_1^{j+1}}(t,\bar x)\right|dt\right)^{\frac{j}{j+1}}\left(\fint_{z(v_i)}^{y(v_i)}|u(t,\bar x)|dt\right)^{\frac{1}{j+1}}d\bar x
\\
&\geq
\frac{1}{|R_k|}\int_{\bar R_k}\sum_l\sum_{i}(y(v_i)-z(v_i))
\frac{2^k}{C_{1D}}
\\
&\geq
\frac{1}{|R_k|}\int_{\bar R_k}\sum_l\frac{|I_l^k(\bar x)|}{4}
\frac{2^k}{C_{1D}}
=
\frac{2^k}{4C_{1D}}
\geq \frac{1}{4C_{1D}}\left|\frac{\partial^{j}u}{\partial x_1^{j}}(v)\right|
\end{aligned}
$$
The third inequality follows from \eqref{onedimodhadinz} using the fact that $\bar x \in B(u_i,\delta_k)$.
The last inequality holds for $v\in E_k$.

Therefore on $\mathbb R^n$ it holds
$$
\left|\frac{\partial^{j}u}{\partial x_1^{j}}\right|
\leq
4{C_{1D}}
\left(
\sum_{k\in\mathbb Z}\chi_{R_k}
\fint_{R_k}\left|\frac{\partial^{j+1} u}{\partial x_1^{j+1}}\right|dx
\right)
^{\frac{j}{j+1}}
\left(
\sum_{k\in\mathbb Z}\chi_{R_k}
\fint_{R_k}|u(x)|dx
\right)^{\frac{1}{j+1}}
$$

\end{proof}

\section{Counter-example}

Now let us remind that Hardy-Littlewood maximal operator is defined by
$$
Mu(x):=\sup_{B\ni x}\frac{1}{|B|}\int_B|u(y)|dy
$$
where $B$ stands for a ball (alternative definition using cubes is also possible).

Agnieszka Kalamajska, in her paper showed the estimate
$$
\left|\nabla^j u(x)\right|\le
M(\nabla^j u)(x)\leq C \left(M(\nabla^k u)(x)\right)^{j/k}(Mu(x))^{1-j/k}.
$$
However, this estimate cannot be extended to pure partial derivatives, which our new approach allows. Let us show by a simple counter-example that the following conjecture is false.
\begin{equation}\label{IMDE}
(\exists C>0)(\forall u\in \mathcal{C}^k_c(\er^n)) \left|\frac{\partial^j u}{\partial x^j}{}\right|\leq C M\left(\frac{\partial^k u}{\partial x^k}\right)^{j/k}(Mu)^{1-j/k}
\end{equation}
And in fact it shows that we cannot expect the family of sets, for which the intermediate derivative is bounded by averaging operators, to be a ball-like structure. 

\begin{Example}
For simplicity consider the case of $j=1, k=2, n=2$. Let $\varphi\in \mathcal{C}^2_c(\er)$ be a function, such that
$$
\varphi(x)=x \textup{ on }(-1/2,1/2),
$$
moreover let 
$$
\textup{supp}(\varphi)\subset \left[-1,1\right]
$$
Let $\eta\in \mathcal{C}_c^\infty(\er)$ be a function with $\supp(\eta)\subset [-1,1]$ and $\eta(0)=1, 0\leq \eta\leq 1$, and $\eta$ icreasing on $[-1,0]$ and decreasing on $[0,1]$. Set 
$$
\eta_n(y):=\eta(ny)
$$
And set
$$
u_n(x,y)=\varphi(x)\eta_n(y)
$$
Now 
$$
\frac{\partial u_n}{\partial x}(0,0)=1.
$$
Observe, the following inequality
$$
\eta_{n+1}\leq \eta_n
$$
Hence
$$
u_{n+1}(x,y)\leq u_{n}(x,y)
$$
Therefore
$$
Mu_n\leq Mu_1\leq \|u_1\|_{\infty}=\|\varphi\|_\infty
$$

For arbitrary $r\leq \frac{1}{2}$
$$
\fint_{B(0,r)}\partial^2_{xx} u_n dx dy = 0
$$
and for arbitrary $r>\frac{1}{2}$ we can estimate
$$
\begin{aligned}
\fint_{B(0,r)}|\partial^2_{xx} u_n| dx dy
&=
\frac{1}{\omega_2r^2}\int_{B(0,r)}|\partial^2_{xx} u_n |dx dy
\\
&\leq
\frac{1}{\omega_2}\int_{[-r,r]\times[-r,r]}|\partial^2_{xx} u_n| dx dy
\\
&=
\frac{1}{\omega_2}\int_{[-r,r]}|\varphi''(x)| dx \int_{[-r,r]} |\eta(ny)| dy
\\
&\leq
\frac{1}{\omega_2n}\|\varphi''\|_1 \|\eta\|_1
\end{aligned}
$$

Therefore

$$
M(\partial_{xx}u)\leq \frac{1}{\omega_2n}\|\varphi''(x)\|_1 \|\eta\|_1
$$
Overall 
$$
(M\partial^2_{xx}u_n)^{\frac{1}{2}}(Mu_n)^{\frac{1}{2}}\leq \frac{1}{\sqrt n}\sqrt{\frac{\|\varphi''\|_1 \|\eta\|_1}{\omega_2}}\|\varphi\|_\infty
$$
Note that the expression tends to $0$ as $n\to\infty$. Hence the inequality (\eqref{IMDE}) cannot hold.

\end{Example}

\section{Open problems}
\begin{Pro}
It was proved that for arbitrary r.i. spaces $X,Y$ the following estimate holds
$$
\|\partial^j_{x_i}u\|_{X^{k/j}Y^{1-k/j}}\leq C \|\partial^k_{x_i}u\|^{k/j}_X\|u\|^{1-k/j}_{Y},
$$
\end{Pro}
where the space on the left-hand side is the space given by Calderon-Lozanovski construction. Same construction can be done for an arbitrary Banach function space. The question is: Is the inequality still valid even for spaces that are not rearrangement-invariant?

\begin{Pro}
    It is likely that the following hypotheses holds. For arbitrary natural numbers $0<j<k$ and $u\in\mathcal{C}^{k}_c(\er^n)$, there exists a family of sets such that 
    $$
    \left|\frac{\partial^j u}{\partial x_i^j}(x)\right|^k\leq CA_{\mathcal{S}}\left(\frac{\partial^k u}{\partial x_i^k}\right)^{j}(x)A_{\mathcal{S}}(u)^{k-j}(x)
    $$
\end{Pro}

Note that it is enough to upgrade the one-dimensional lemma \ref{ODE} to the general case replacing $j+1$ by general natural $k>j$. However, this requires a brand new design of the intervals.


\begin{thebibliography}{9}

\bibitem{BS}
Colin Bennett and Robert~C. Sharpley.
\newblock {\em Interpolation of operators}.
\newblock Academic Press, 1988.

\bibitem{Gag}
Emilio Gagliardo.
\newblock Ulteriori propriet\`a di alcune classi di funzioni in pi\`u variabili.
\newblock {\em Ricerche Mat.}, 8:24, 1959.

\bibitem{Kal}
Agnieszka Ka{\l}amajska.
\newblock Pointwise multiplicative inequalities and {N}irenberg type estimates in weighted {S}obolev spaces.
\newblock {\em Studia Math.}, 108(3):275--290, 1994.

\bibitem{Lady}
Olga~A. Lady{\v{z}}enskaja.
\newblock Solution ``in the large'' to the boundary-value problem for the {N}avier--{S}tokes equations in two space variables.
\newblock {\em Soviet Physics. Dokl.}, 123(3):1128--1131 (427--429 {\em Dokl. Akad. Nauk SSSR}), 1958.

\bibitem{Landau1913}
Edmund Landau.
\newblock Einige {U}ngleichungen f{\"u}r zweimal differenzierbare {F}unktionen.
\newblock {\em Proceedings of the London Mathematical Society}, 13:43--49, 1913.

\bibitem{LRS2}
Karol Le{\'s}nik, Tom{\'a}{\v{s}} Roskovec, and Filip Soudsk{\'y}.
\newblock Gagliardo--{N}irenberg inequality via a new pointwise estimate.
\newblock {\em Journal of Functional Analysis}, 289(7):110996, 2025.

\bibitem{N}
John Nash.
\newblock Continuity of solutions of parabolic and elliptic equations.
\newblock {\em American Journal of Mathematics}, 80(4):931--954, 1958.

\bibitem{Ni}
Louis Nirenberg.
\newblock On elliptic partial differential equations.
\newblock {\em Annali della Scuola Normale Superiore di Pisa-Scienze Fisiche e Matematiche}, 13(2):115--162, 1959.

\end{thebibliography}
\end{document}